\documentclass[12pt,a4paper]{article}
\usepackage[utf8]{inputenc}
\usepackage[english]{babel}
\usepackage[T1]{fontenc}
\usepackage[pdftex]{hyperref}
\usepackage{amsmath}
\usepackage{amsthm}
\usepackage[square,numbers]{natbib}
\usepackage{soul}
\usepackage{hyperref}
\usepackage{url}
\hypersetup{
colorlinks=true,
linkcolor=black,
filecolor=magenta,
urlcolor=blue,
pdftitle={thèse}
bookmarks=true,
citecolor=black
}
\usepackage{verbatim}
\usepackage{stmaryrd}
\usepackage{amsfonts}
\usepackage{amssymb}
\usepackage{mathrsfs}
\usepackage{setspace}
\usepackage{pdfpages}
\usepackage[all]{xy}
\usepackage[OT2,T1]{fontenc}
\DeclareSymbolFont{cyrletters}{OT2}{wncyr}{m}{n}
\newcommand{\Sum}{\displaystyle \sum}
\newcommand{\Int}{\displaystyle \int}

\newcommand{\F}{\mathbb{F}_{p}}
\newcommand{\Q}{\mathbb{Q}}
\newcommand{\Z}{\mathbb{Z}}
\newcommand{\R}{\mathbb{R}}
\newcommand{\Pj}{\mathbb{P}}

\DeclareMathSymbol{\Sha}{\mathalpha}{cyrletters}{"58}

\newtheorem{def }{Definition}
\newtheorem{defp}[def ]{Definition/Proposition}
\newtheorem{prop}[def ]{Proposition}
\newtheorem{theo}[def ]{Theorem}
\newtheorem*{theo*}{Theorem}
\newtheorem{lem }[def ]{Lemma}

\theoremstyle{definition}
\newtheorem*{rem}{Remark}

\usepackage{xcolor}
\date{}
\author{Valentin Petit}
\title{Non-divisible point on a two-parameter family of elliptic curves}
\begin{document}
\maketitle
\begin{abstract} Let $n$ be a positive integer and $t$ be a non-zero integer. We consider the two-parameter family of elliptic curves over $\Q$ given by 
$$
\mathcal{E}_n(t)\colon y^2=x^3+tx^2-n^2(t+3n^2)x+n^6.
$$
We prove a result of non-divisibility of the point $(0,n^3) \in \mathcal{E}_n(t)(\Q)$ whenever $t$ is sufficiently large compared to $n$ and $t^2+3n^2t+9n^4$ is squarefree. Our work extends to this family of elliptic curves a previous study of Duquesne mainly stated for $n=1$ and $t>0$. \\ 



\end{abstract}
\section{Introduction}
We are concerned with proving the non-divisibility of a point on a family of elliptic curves defined over $\Q$ with two integer parameters. 
This family generalizes Washington's family \cite{washington1987class} which is connected to simplest cubic fields. 
Let $n$ be a positive integer. We consider the elliptic curve over $\Q(T)$ (which can be seen as an elliptic surface over $\Pj^1(\Q)$) given by 
\begin{equation}\label{eq:E}
 \mathcal{E}_n \colon y^2=x^3+Tx^2-n^2(T+3n^2)x+n^6.
\end{equation}
It is a special case of an elliptic surface studied  by Bettin, David and Delaunay \citep{bettin2018non}. The case $n=1$ is precisely Washington's family studied by Washington \citep{washington1987class} and Duquesne \citep{duquesne2001integral}. In \citep{bettin2018non} the authors obtain a formula for the average root number of the elliptic curves obtained by specializing $\mathcal{E}_n$ at 
$T=t\in \Q$. We denote by $\mathcal{E}_n(t)$ this specialization. \\
Another  elliptic curve over $\Q(T)$ derived from their family is 
\begin{equation}
\mathcal{F}_n : y^2=x^3+Tx^2+n^2(T-3n^2)x-n^6.
\end{equation}
Both elliptic curves $\mathcal{E}_n$ and $\mathcal{F}_n$ are related. Indeed, for all positive integers $n$ and all integers $t$, the curves $\mathcal{E}_n(t-3n^2)$ and $\mathcal{F}_n(t)$ are isomorphic over $\Q$. We will focus on the curve $\mathcal{E}_n$. 

Let $t \in \Z$ Theorem~5.7 of Duquesne in \citep{duquesne2001integral} asserts that if $n=1$, $t$ is positive and $t^2+3t+9$ is squarefree, then the point $(0,1)$ is not divisible on $E_1(t)$. We generalize the result as follows for the integral point $(0,n^3)$.
\begin{theo}\label{theo1}
Let $n$ be a positive integer. Suppose that $t \geq \max(100n^2,n^4)$ or $t \leq \min(-100n^2,-2n^4)$, and $t^2+3n^2t+9n^4$ is squarefree. Then the point $(0,n^3)$ is not divisible on $\mathcal{E}_n(t)$.
\end{theo}
As far as we know, there  exists no similar statement for other two-parameter families of elliptic curves over $\Q$.

A consequence of Theorem \ref{theo1} is that the point $(0,n^3)$ can be taken as an element of a system of generator for $E_n(t)(\Q)$ when $t$ is sufficiently large. When the rank of $\mathcal{E}_n(t)$ is one, it is then possible to compute the analytic order of the Tate-Shafarevich group by using the BSD conjecture as in \citep{delaunay2003numerical}. 
In particular, we note that except when $n=1$, the root number is not constant and not equidistributed which motivates the study of this type of family (\citep{bettin2018non}). For example when $n=2$, the average root number is $-\frac{1}{2}$, which means 75 $\% $ of specializations have odd rank under the parity conjecture.

For all positive integer $n$, the elliptic curve $\mathcal{E}_n$ has rank one and the point $(0,n^3)$ has infinite order in $\mathcal{E}_n(\Q(T))$ (\citep{bettin2018non}, Theorem 1).
Silverman specialization Theorem on elliptic surfaces in \citep{silverman1985divisibility} has the following consequence.
Since the specialization map $\sigma_t \colon \mathcal{E}_n \rightarrow \mathcal{E}_n(t)$ is injective for all but only finitely many $t \in \Pj^1(\Q)$, Theorem \ref{theo1} implies that the point $(0,n^3)$ is a generator of $\mathcal{E}_n$. As regards the non-divisibility (see \citep{silverman1985divisibility}) of the image of $\mathcal{E}_n$ by the specialization map, Theorem 2 in \citep{silverman1985divisibility} asserts that the set of $t \in \Z$ such that $\sigma_t(\mathcal{E}_n)$ is non-divisible has density one. Our theorem holds provided that $t^2+3n^2t+9n^4$ is squarefree: it represents a set of $t \in \Z$ of density less than one, which gives a less strong result than predicted by Theorem 2 in \citep{silverman1985divisibility} but is completely explicit.
Furthermore, for all positive integer $n$, there exists a counterexample to non-divisibility when $t^2+3n^2t+9n^4$ is not squarefree. If $t=5n^2$, we have $(0,n^3)=3(-4n^2,7n^3)$ which shows that the squarefree hypothesis cannot be removed from Theorem \ref{theo1}. However, note that we have not found another example where the point $(0,n^3)$ is the multiple of some point.

The strategy for proving Theorem \ref{theo1} is similar to Duquesne's with an additional specific and careful treatment of the two parameters (we also notice that a small issue seems to occur in Section 5D of  \citep{duquesne2001integral} during the computation of the local contribution of the height, which we can correct here). Suppose that there exist an integer $k \geq 2$ and a point $P \in \mathcal{E}_n(t)(\Q)$ such that $kP=(0,n^3)$. The main idea is to minimize 
such an integer $k$. The strategy is to find a lower bound for the canonical height $\widehat{h}(P)$ of $P$, an upper bound for $\widehat{h}((0,n^3))$ (Section \ref{Pht}), and to obtain a contradiction. To do this, we split the canonical height into local contributions (Section \ref{Pht}) and approximate the periods of $\mathcal{E}_n(t)$ (Section~\ref{Pap}).

\section{Generalities about the family}\label{Gfam}
Let $n$ be a positive integer and $t$ a non-zero integer. We consider the elliptic curve 
\begin{equation}\label{equE}
E : y^2=x^3+tx^2-n^2(t+3n^2)x+n^6,
\end{equation}
which was denoted by $\mathcal{E}_n(t)$ in the introduction. We set $\delta=t^2+3n^2t+9n^4$. 
The discriminant and $j$-invariant of the elliptic curve $E$ are given by
\begin{equation*}
\begin{array}{cll}
\Delta &=& 16n^4 \delta^2,\vspace{0,15cm}\\
 
j&=&\frac{256}{n^4}\delta, \vspace{0,15cm}\\
c_4&=&16 \delta.
 
\end{array}
\end{equation*}
Let $f(x)$ be the polynomial $$f(x)=x^3+tx^2-n^2(t+3n^2)x+n^6.$$ 
Its discriminant is $n^4 \delta^2$, which is positive, so $f$ has  three real roots denoted by $\alpha_1 <\alpha_2 <\alpha_3$. They satisfy $\alpha_1<0<\alpha_2<n^2<\alpha_3$. 
The polynomial $f$ is irreducible over $\Q$ if $\delta$ is squarefree. Indeed, it suffices to see that the polynomial $h(x)=27 f\left(\dfrac{x-t}{3}\right)=x^3+3\delta x+\delta(2t+3n^2)$ is irreducible over $\Q$ by Eisenstein's criterion.

We denote by $E_0(\R)$ the connected component of the identity element and by $E(\R)-E_0(\R)$ the bounded connected component of $E(\R)$. Recall that $E_0(\R)$ is a subgroup of $E(\R)$ and that the sum of two points in $E(\R)-E_0(\R)$ lies in $E_0(\R)$. The integral point $(0,n^3) $ belongs to $E(\Q)$ for all $t$ and $n$.
More precisely, the fact that $\alpha_2>0$ implies that the point $(0,n^3)$ belongs to $E(\Q)-E_0(\Q)$ for all $t \in \Z_{\neq 0}$ and all positive integer $n$.

We note that $(0,n^3)$ is not a torsion point of $E$ when $\delta$ is squarefree. Indeed, since $(0,n^3) \in E(\Q)-E_0(\Q)$, the order of $(0,n^3)$ cannot be odd if it is finite. Moreover, since the polynomial $f(x)$ is irreducible over $\Q$ with the assumption $\delta$ squarefree, there is no torsion point of even order. Hence the point $(0,n^3)$ has infinite order.

Throughout the article we assume that $\delta$ is squarefree, which implies that $n$ and $t$ are coprime. This condition will play a key role later. 
Moreover we need \eqref{equE} to be a minimal Weierstrass equation for $E$ which, by Tate's algorithm \cite[IV.9]{silverman2013advanced}, occurs if we assume also that $t$ is not congruent to $1$ modulo $4$ when $4\mid n$. The minimal Weierstrass equation will be necessary to compute the non-archimedean local contributions.

For the case $4 \mid n$ and $t\equiv 1 \, [4]$, write $n=4m$ and $t=4k+1$ with $m \in \Z_{>0}$ and $k \in \Z$. The elliptic curve $E$ is then isomorphic over $\Q$ to the curve
\begin{equation}\label{4m}
E':y^2+xy=x^3+kx^2-m^2(4k+1-48m^2)x+64m^6.
\end{equation}
The curve $E'$ is connected to $E$ by the change of variables $x=4x'$, $y=8y'+x'$.
Note that (\ref{4m}) is a minimal model for $E'$.
The change of variables maps the point $(0,n^3) \in E(\Q)$ to the point $(0,8m^3)\in E'(\Q)$. Therefore, if we want to prove the non-divisibility of $(0,n^3)$ on $E(\Q)$, it suffices to prove the non-divisibility of $(0,8m^3)$ on $E'(\Q)$.

\section{Approximation of periods}\label{Pap}

The goal of this section is to approximate the real period $ \omega_1$ and the imaginary period $\omega_2$ of $E$. In order to compute $\omega_1$ and $\omega_2$, the Weierstrass equation defining $E$ does not need to be minimal.
Let $n$ be a fixed positive integer, the elliptic curve $E: y^2=x^3+tx^2-n^2(t+3n^2)x+n^6$ is isomorphic over $\Q$ to the curve 
$$y^2=4g(x),$$
where $g(x)=x^3-\frac{1}{3}\delta x+\frac{1}{27}(2t+3n^2)\delta$. 
The real roots  $e_1,e_2,e_3$ of $g$ are given by $e_i=\alpha_i+\dfrac{t}{3}$ for all $i \in \{1,2,3\}$ and the periods $\omega_1$ and  $\omega_2$ of $E$ are (see \cite[7.3.2]{cohen2008number})
\begin{equation*}
\omega_1=\Int_{e_1}^{e_2}\dfrac{dx}{\sqrt{g(x)}} \quad \in \R, \hspace{0.5cm} \omega_2=-\Int_{e_2}^{e_3}\dfrac{dx}{\sqrt{g(x)}} \quad \in i\R.
\end{equation*}\label{ap1}
\subsection{The case  $t>0$ 
}
When $t>0$, the real roots of $g$ satisfy $e_1<0<e_2<e_3$ because $g(0)>0$. 
A straightforward study of the function $g$ gives for $t \geq 3n^2$,
\begin{equation}\label{ineqrt1}
 \begin{aligned}
-\frac{2}{3}t-n^2-2\frac{n^4}{t} &\leq e_1 \leq -\frac{2}{3}t-n^2-\frac{n^4}{t},\\
\frac{t}{3}&\leq e_2 \leq \frac{t}{3}+\frac{n^4}{t},\\
\quad  \frac{t}{3}+n^2 &\leq e_3 \leq  \frac{t}{3}+n^2+\frac{n^4}{t}.
  \end{aligned}
\end{equation}
\begin{lem } \label{l1}

Let $n$ be a fixed positive integer. As $t\rightarrow +\infty$, 
we have $\dfrac{\omega_2}{i} \sim \dfrac{\pi}{\sqrt{t}}$. Moreover for $t \geq 100n^2$, we have 
\begin{equation*}
\dfrac{3.11}{\sqrt{t}}\leq \dfrac{\omega_2}{i} \leq \dfrac{3.15}{\sqrt{t}}.
\end{equation*}
\end{lem }
\begin{proof}
We note that  $\dfrac{\omega_2}{i}=\Int_{e_2}^{e_3}\dfrac{dx}{\sqrt{(x-e_1)(x-e_2)(e_3-x)}}$. If $x \in [e_2,e_3]$, we have  $t+n^2+\dfrac{n^4}{t} \leq x-e_1 \leq t+2n^2+\dfrac{3n^4}{t}$ by \eqref{ineqrt1}. So
$$\dfrac{1}{\sqrt{t+2n^2+\frac{3n^4}{t}}} \leq \dfrac{1}{\sqrt{x-e_1}} \leq \dfrac{1}{\sqrt{t+n^2+\frac{n^4}{t}}},$$
and then 
\begin{equation}\label{tri1}
\dfrac{\Int_{e_2}^{e_3}\dfrac{dx}{\sqrt{(x-e_2)(e_3-x)}}}{\sqrt{t+2n^2+\frac{3n^4}{t}}} \leq \dfrac{\omega_2}{i} \leq \dfrac{\Int_{e_2}^{e_3}\dfrac{dx}{\sqrt{(x-e_2)(e_3-x)}}}{\sqrt{t+n^2+\frac{n^4}{t}}},
\end{equation}
where $\Int_{e_2}^{e_3}\dfrac{dx}{\sqrt{(x-e_2)(e_3-x)}}=\pi$. Moreover, as $t \rightarrow +\infty$ the left-hand side and the right-hand side of \eqref{tri1} are both equivalent to $\dfrac{\pi}{\sqrt{t}}$.
Finally, when $t \geq 100n^2$, we derive from (\ref{tri1})
$$\dfrac{3.11}{\sqrt{t}} \leq \dfrac{\omega_2}{i} \leq \dfrac{3.15}{\sqrt{t}}.$$ 
\end{proof}

\begin{lem }\label{l2} For $t \geq 100n^2$, we have
$$
\dfrac{1.88+0.99 \log\left(\frac{t}{n^2}\right)}{\sqrt{t}} \leq \omega_1 \leq \dfrac{5.35+1.23 \log\left(\frac{t}{n^2}\right)}{\sqrt{t}}.
$$
\end{lem }
\begin{proof}
We split the integral into two parts
$$\omega_1^{-}=\Int_{e_1}^{0}\dfrac{dx}{\sqrt{(x-e_1)(e_2-x)(e_3-x)}}, \omega_1^{+}=\Int_{0}^{e_2}\dfrac{dx}{\sqrt{(x-e_1)(e_2-x)(e_3-x)}}.$$
 First we consider $\omega_1^{-}$. If $x \in [e_1,0]$, we have $\frac{t}{3} \leq e_2-x \leq t+n^2+\frac{3n^4}{t}$ and $\frac{t}{3}+n^2 \leq e_3-x \leq t+2n^2+\frac{3n^4}{t}$ by \eqref{ineqrt1}. So we get the lower bound
$$
\begin{array}{cl}
 \omega_1^{-} \sqrt{\left(t+n^2+\dfrac{3n^4}{t}\right)\left(t+2n^2+\dfrac{3n^4}{t}\right)} &\geq  \Int_{e_1}^{0}\frac{dx}{\sqrt{x-e_1}} \vspace{0.2cm}\\
& \geq 2\sqrt{-e_1}  \vspace{0.25cm}\\
& \geq 2\sqrt{\frac{2}{3} t+n^2+\frac{n^4}{t}}. \vspace{0.25cm}\\
\end{array}
$$
Furthermore using \eqref{ineqrt1}, we get the upper bound 
$$
\begin{array}{rlll}
  \omega_1^{-} \leq & \dfrac{\Int_{e_1}^{0}\frac{dx}{\sqrt{x-e_1}}}{\sqrt{\dfrac{t}{3}\left(\dfrac{t}{3}+n^2 \right)}}
  &\leq \dfrac{2\sqrt{-e_1}}{\sqrt{\dfrac{t}{3}\left(\dfrac{t}{3}+n^2 \right)}} &
 \leq  \dfrac{6}{\sqrt{t}}\sqrt{\dfrac{2}{3}+\dfrac{n^2}{t}+\dfrac{2n^4}{t^2}}.
\end{array}
$$
Thus when $t \geq 100n^2$, we obtain 
$$ \dfrac{1.60}{\sqrt{t}} \leq \omega_1^{-} \leq \dfrac{4.94}{\sqrt{t}}.
$$
Now we consider $\omega_1^{+}$. If $x \in [0,e_2]$, we have $\frac{2}{3}t+n^2+\frac{n^4}{t} \leq x-e_1  \leq t+n^2+\frac{n^4}{t}$ by \eqref{ineqrt1}. So we have

$$
\dfrac{J}{\sqrt{t+n^2+\frac{3n^4}{t}}} \leq \omega_1^{+} \leq \dfrac{J}{\sqrt{\frac{2}{3} t+n^2+\frac{n^4}{t}}},
$$ 
with
$$
J=\Int_{0}^{e_2}\dfrac{dx}{\sqrt{(e_2-x)(e_3-x)}}=\log \left(\dfrac{\sqrt{e_3}+\sqrt{e_2}}{\sqrt{e_3}-\sqrt{e_2}} \right).
$$
Moreover by \eqref{ineqrt1}, we have
$$
\begin{array}{rcl}
\dfrac{\frac{4}{3}t+n^2}{n^2+\frac{n^4}{t}}  &\leq \dfrac{\sqrt{e_3}+\sqrt{e_2}}{\sqrt{e_3}-\sqrt{e_2}} \leq & \dfrac{\frac{2}{3}t+n^2+\frac{2n^4}{t}+2\sqrt{(\frac{t}{3}+n^2+\frac{n^4}{t})(\frac{t}{3}+\frac{n^4}{t})}}{n^2-\frac{n^4}{t}}, 
\end{array}
$$
which implies 
\begin{eqnarray*}
\dfrac{\frac{4}{3}t+n^2}{n^2+\frac{n^4}{t}}  &\leq \dfrac{\sqrt{e_3}+\sqrt{e_2}}{\sqrt{e_3}-\sqrt{e_2}} \leq & \dfrac{\frac{4}{3}t+3n^2+\frac{4n^4}{t} }{n^2-\frac{n^4}{t}}.
\end{eqnarray*}
So we get 
$$
\dfrac{1}{\sqrt{t+n^2+\frac{3n^4}{t}}}\log \left(\dfrac{\frac{4}{3}t+n^2}{n^2+\frac{n^4}{t}}\right) \leq \omega_1^{+} \leq \dfrac{1}{\sqrt{\frac{2}{3}t+n^2+\frac{n^4}{t}}}\log \left(\dfrac{\frac{4}{3}t+3n^2+\frac{4n^4}{t}}{n^2-\frac{n^4}{t}} \right).
$$
Assume now that $t \geq 100n^2$. We derive
\begin{align*}
\omega_1^{+}&\leq \dfrac{1}{\sqrt{\frac{2}{3}t+n^2+\frac{n^4}{t}}} \left( \log(t)+\log\left(\frac{4}{3}\right) + \log \left(1+ \frac{3n^2}{t}+\frac{4n^4}{t^2}
\right)-\log\left(n^2-\frac{n^4}{t}\right)\right) \\
& 
\leq \dfrac{0.41+1.23 \log\left(\frac{t}{n^2}\right)}{\sqrt{t}}
\end{align*}
and 
\begin{align*}
\omega_1^{+} \geq \dfrac{1}{\sqrt{t+n^2+\frac{3n^4}{t}}}\log \left(\dfrac{4t}{3(n^2+\frac{n^4}{t})}\right) \geq \dfrac{0.28+0.99 \log \big( \frac{t}{n^2}\big)}{\sqrt{t}}.
\end{align*}
Finally we obtain the desired conclusion.


\end{proof}
\begin{rem}
A numerical analysis suggests that $\omega_1$ should be equivalent to  $\dfrac{\log\left(\frac{t}{n^2}\right)}{\sqrt{t}}$ as $t \rightarrow +\infty$.
\end{rem}
\subsection{The case $t<0$ 
}
Once again we use the fact that $E$ is isomorphic to the curve $y^2=4g(x)$.
A straightforward study of the function $g$ gives the following estimates when $t \leq -3n^2$:
\begin{equation}\label{ineqrt2}
\begin{array}{rcl}
\dfrac{t}{3}+\dfrac{2n^4}{t}&\leq e_1 \leq & \dfrac{t}{3}+\dfrac{n^4}{t}, \\
\dfrac{t}{3}+n^2+\dfrac{2n^4}{t} &\leq e_2 \leq & \dfrac{t}{3}+n^2+\dfrac{n^4}{t}, \\
-\dfrac{2t}{3}-n^2 &\leq e_3 \leq & -\dfrac{2t}{3}.
\end{array}
\end{equation}
\begin{lem }\label{l3}
 As $t \rightarrow -\infty$, we have $\omega_1 \sim \dfrac{\pi}{\sqrt{|t|}}$. Moreover, for $t \leq -100n^2$ we have 
\begin{equation*}
\dfrac{3.14}{\sqrt{|t|}} \leq \omega_1 \leq \dfrac{3.15}{\sqrt{|t|}}.
\end{equation*}
\end{lem }
\begin{proof}
If $x \in [e_1,e_2]$, then $|t|-2n^2+\frac{n^4}{|t|} \leq e_3-x \leq |t|+\frac{2n^4}{|t|}$ by \eqref{ineqrt2}. So
$$
\dfrac{1}{\sqrt{|t|+\frac{2n^4}{|t|}}} \leq \dfrac{1}{\sqrt{e_3-x}} \leq \dfrac{1}{\left|\sqrt{|t|}-\frac{n^2}{\sqrt{|t|}}\right|}.$$
We obtain
$$
\dfrac{J}{\sqrt{|t|+\frac{2n^4}{|t|}}} \leq \omega_1 \leq \dfrac{J}{\left|\sqrt{|t|}-\frac{n^2}{\sqrt{|t|}}\right|},$$
with
$$
J=\Int_{e_1}^{e_2}\dfrac{dx}{\sqrt{(x-e_1)(e_2-x)}}=\pi.
$$
Since $\dfrac{\pi}{\sqrt{|t|+\frac{2n^4}{|t|}}}$ and $\dfrac{\pi}{\sqrt{|t|}-\frac{n^2}{\sqrt{|t|}}}$ are both equivalent to $\dfrac{\pi}{\sqrt{|t|}}$ as $t \rightarrow -\infty$, we deduce $\omega_1 \underset{t \rightarrow -\infty}{\sim} \dfrac{\pi}{\sqrt{|t|}}$.
Furthermore if $t \leq -100n^2$, we get
$$\dfrac{3.14}{\sqrt{|t|}} \leq \omega_1 \leq \dfrac{3.15}{\sqrt{|t|}}.$$
\end{proof}
\begin{lem }\label{l4}
If $t \leq -100n^2$, we have 
\begin{equation*}
\dfrac{\omega_2}{i} \geq \dfrac{0.39+\log\left(\frac{|t|}{n^2}\right)}{\sqrt{|t|}}.
\end{equation*}
\end{lem }
\begin{proof}
Note that $e_2<0<e_3$ as $g(0)<0$ for $t \leq 100n^2$. To approximate $\omega_2$, we need to split the integral into two parts:
$$
\begin{array}{cccc}
\dfrac{\omega_2}{i}= &\underbrace{\Int_{e_2}^{0}\dfrac{dx}{\sqrt{(x-e_1)(x-e_2)(e_3-x)}}}&+&\underbrace{\Int_{0}^{e_3}\dfrac{dx}{\sqrt{(x-e_1)(x-e_2)(e_3-x)}}} .\\
& W^{-} & & W^{+}
\end{array} 
$$
We begin with estimating $W^{-}$. If $x \in [e_2,0]$, we get by \eqref{ineqrt2} $$-\dfrac{2}{3}t-n^2 \leq e_3-x \leq -t-n^2-\dfrac{2n^4}{t}.$$ So 
$$\dfrac{L}{\sqrt{-t-n^2-\frac{2n^4}{t}}} \leq W^{-} \leq \dfrac{L}{\sqrt{-\frac{2}{3}t-n^2}},
$$
with
$$
L=-\Int_{0}^{e_2}\dfrac{dx}{\sqrt{(x-e_1)(x-e_2)}}=-\log(e_2-e_1)+\log \big(-e_1-e_2+ \sqrt{e_1 e_2} \big).
$$
Moreover by \eqref{ineqrt2}, we have
$$
-\log \left(n^2-\dfrac{n^4}{t} \right)+\log \left(-\dfrac{2}{3}t-n^2-\dfrac{2n^4}{t} \right) \leq L.
$$
For $t \leq -100n^2$, we get
$$
\begin{array}{lcl}
W^{-}\sqrt{-t-n^2-\frac{2n^4}{t}}  &\geq &-\log \left(n^2- \dfrac{n^4}{t} \right)+\log \left(-\dfrac{2}{3}t-n^2-\dfrac{2n^4}{t} \right) \vspace{0,3cm} \\
& \geq& -2 \log(n)-\log\left(1-\dfrac{1}{100}\right)+\log(|t|)+\log\left(\dfrac{2}{3}-\dfrac{1}{100}\right),
\end{array}
$$
from which we derive
\begin{equation}\label{W-}
W^{-} \geq \dfrac{\log \left(\frac{|t|}{n^2}\right)-0.42}{\sqrt{|t|}}.
\end{equation}
If $x \in [0,e_3]$, we obtain by \eqref{ineqrt2}
\begin{equation*}
\begin{aligned}
-\frac{t}{3}-\frac{n^4}{t}& \leq x-e_1 \leq -t-\frac{2n^4}{t}, \\
-\frac{t}{3}-n^2-\frac{n^4}{t} &\leq  x-e_2 \leq  -t-n^2-\frac{2n^4}{t}.
\end{aligned}
\end{equation*}
Thus
$$
\begin{array}{rcl}
 W^{+} \sqrt{\left(-t-\frac{2n^4}{t}\right)\left(-t-n^2-\frac{n^4}{t}\right)}\geq 2 \sqrt{e_3},
\end{array}
$$
from which we derive
\begin{align*}
W^{+} \geq \dfrac{\sqrt{-\frac{2}{3}t-n^2}}{\sqrt{(-t-\frac{2n^4}{t})(-t-n^2-\frac{n^4}{t})}} \geq \dfrac{\sqrt{|t|}\sqrt{\frac{2}{3}-\frac{n^2}{|t|}}}{|t|\sqrt{(1+\frac{2n^4}{|t|^2})^2}}.
\end{align*}
So we obtain, for $t \leq -100n^2$
\begin{equation}\label{W+}
W^{+} \geq \dfrac{0.81}{\sqrt{|t|}}.
\end{equation}
Finally by adding both inequalities \eqref{W-} and \eqref{W+}, we have  
$$
\dfrac{\omega_2}{i} \geq \dfrac{0.39+\log\left(\frac{|t|}{n^2}\right)}{\sqrt{|t|}}.
$$
\end{proof}
\begin{rem}
When $t<0$, we just give a lower bound for $\frac{\omega_2}{i}$ since no upper bound  is required in order estimate the height.
\end{rem}
\section{Estimates on the heights}\label{Pht}

We need to estimate the heights of some points in $E(\Q)$. For this purpose, we will decompose the height as a sum of local contributions. Depending on the conventions, there are different ways to split the canonical height into a sum of local heights. However in this section, a special care was taken during the computations of local heights  to make sure that their sum agrees with the definition of the canonical height.

\subsection{Lower bound for the height}\label{subs:lowerbound}
We want  to prove that there does not exist a point $P=(\alpha, \beta) \in E(\Q)$ and an integer $\ell \geq 2$ such that $\ell P=(0,n^3) $. 
The goal of this part is to find, if such a point exists, a lower bound for the canonical height of $P$. Since $(0,n^3) \in E(\Q)-E_0(\Q)$, the point $P$ has to belong to $E(\Q)-E_0(\Q)$ and $\ell$  must be odd (see Section \ref{Gfam}). 
\begin{lem }\label{entier}
Let $F: y^2+a_1xy+a_3y=x^3+a_2x^2+a_4x+a_6$ be an elliptic curve over $\Q$ with $a_i \in \Z$ for all $i \in \{1,2,3,4,6\}$. Let $P \in F(\Q)$ be a point of infinite order such that $mP$ is an integral point for some integer $m \geq 1$. Then $P$ is an integral point on $F$.
\end{lem }
\begin{proof}
A proof of this lemma is established in \citep{ayad1992points}.
Write $mP=(\alpha,\beta)$ with $\alpha,\beta \in \Z$ and $P=(x,y)$ with $x,y \in \Q$. Let $\psi_m$, $\phi_m$ be the $m^{\mathrm{th}}$-division polynomials. Recall that $\phi_m$ is a monic polynomial in $\Z[X]$ of degree $m^2$ and $\psi_m^2$ is a polynomial in $\Z[X]$ of degree $m^2-1$ (for more details see \cite{ayad1992points}, Section 2). We have 
$$\dfrac{\phi_m(x)}{\psi_m^2(x)}=\alpha.$$
We get that $x$ is a root of the polynomial $\phi_m(X)-\alpha \psi_m^2(X)$. Since $\phi_m(X)-\alpha \psi^2_m(X)$  is monic in $\Z[X]$ and $x \in \Q$, we deduce that $x \in \Z$ and then $y \in \Z$. Hence $P$ is an integral point.
\end{proof}
Lemma~\ref{entier} implies that if $(0,n^3)$ is the multiple of a rational point $P$ then $P \in E(\Z)$.
Moreover the point $(0,n^3)$ is singular modulo $p$ for all primes $p \mid n$. This implies that $P$ is singular modulo $p$ for all $p\mid n$ since the multiple of a non-singular point is always non-singular. 
If $p \mid n$, the only singular point modulo $p$ is $(0,0)$ so we have $p \mid \alpha$. This remark will be used in Lemmas \ref{b+}, \ref{b-} and Proposition \ref{loctr}. 
\begin{lem }\label{b+}
Let $P=(\alpha,\beta) \in E(\Z)-E_{0}(\Z)$ be such that there exists $\ell \geq 2$ with $\ell P=(0,n^3)$. If $t \geq n^4$ then $|\beta| \geq n\sqrt{2t}$.  
\end{lem }
\begin{proof}
Assume that $t \geq n^4$ and $n \geq 2$.
Recall that $\alpha_1, \alpha_2, \alpha_3$ are the real roots of the polynomial $f(x)=x^3+tx^2-n^2(t+3n^2)x+n^6$ with $\alpha_1<0<\alpha_2<\alpha_3$.  By studying  the function  $f$, we have for $t\geq n^4$
$$
-n^2-t-1 < \alpha_1 <-n^2-t \quad\text{and}\quad 0<\alpha_2<1.
$$
Since $P \in E(\Q)-E_0(\Q)$, we have $\alpha_1<\alpha<\alpha_2$, and then $\alpha \in [-n^2-t,-1]$. Moreover since $p \mid \alpha$ for all primes $p \mid n$, we get $\alpha \leq -2$ and then 
$$
\begin{array}{ll}
f(\alpha) &\geq \min(f(-n^2-t),f(-2)) \\
& \geq \min \big(2n^4t + 3n^6,(2n^2 + 4)t + n^6 + 6n^4 - 8 \big) \\
& \geq 2n^2t.
\end{array}
$$
Thus $|\beta| \geq n\sqrt{2t}$. If $n=1$, we have 
$$f(\alpha) \geq \min(f(-t-1),f(-1))=2t+3,
$$
hence $|\beta| \geq \sqrt{2t+3} \geq \sqrt{2t}$.
\end{proof}
\begin{lem }\label{b-}
Suppose that $t \leq -2n^4$. Then
$$|\beta| \geq 
\begin{cases}
 \sqrt{2|t|} & \text{if $n=2$},\\
n\sqrt{|t|} & \text{if $n\geq 3$}.
\end{cases}$$
\end{lem }
\begin{proof}
We use an argument similar to the case $t \geq n^4$ (Lemma~\ref{b+}).
\end{proof}
\begin{rem}
Assume $n=1$. If $t\leq -2$, the point $(0,1)$ is the only integral point in $E(\Q)-E_0(\Q)$ because $-1<\alpha_1<\alpha_2$. Hence $(0,1)$ is not divisible. If $t=-1$, it is easy to prove that this point is not divisible as well. Therefore when $t$ is negative, we can assume that $n \geq 2$.  
\end{rem}
In order to find a lower bound for the canonical height of an integral point of $E$, we use the decomposition of the canonical height into local contributions (see \citep[Theorem 5.2]{silverman1988computing}, \citep[7.5.7]{cohen2013course}). Let $P=(\alpha,\beta) \in E(\Q)$  be an integral point, we have
$$
\widehat{h}(P)=\Sum_{p \leq \infty} \lambda_p(P), \ 
$$
where the sum runs through the places of $\Q$. 
We recall the definitions of $\lambda_p$ for finite places $p$ in Definition/Proposition~\ref{locnarch} and $\lambda_\infty$ in Definition/Proposition~ \ref{defp-htlocalefinie}.
For this, we set 
\begin{align*}
A&=3\alpha^2+2t\alpha-n^2(t+3n^2), \\
B&=2 \beta, \\
C&=3\alpha^4+4t\alpha^3-6n^2(t+3n^2)\alpha^2+12n^6\alpha-n^4(t^2+2n^2t+9n^4), \\
D&=\gcd(A,B), \\
c_4&=16 \delta.
\end{align*}

\begin{defp}\label{locnarch}(\citep[Theorem 5.2]{silverman1988computing},\citep[7.5.6]{cohen2013course})
Let $p$ be a prime number.
The local non-archimedean contribution $\lambda_p(P)$ is non-zero only if $p \mid D$. If $p \mid D$, we set $m_p=\min \left(\frac{v_p(\Delta)}{2},v_p(B)\right)$ and the local contribution at $p$ is then given by
$$
\lambda_p(P)= 
\left\{
\begin{array}{ll}
-\dfrac{m_p( v_p(\Delta)-m_p)}{2 v_p(\Delta)}\log(p) & $if $p \nmid c_4, \vspace{0,3cm} \\
-\dfrac{v_p(B)}{3}\log(p) & $if $p \mid c_4$ and $v_p(C) \geq 3v_p(B), \vspace{0,3cm}\\
-\dfrac{v_p(C)}{8}\log(p) & $otherwise$,
\end{array}
\right.
$$
where $v_p$ is the $p$-adic valuation.
\end{defp}

\begin{lem }\label{canh1}
Assume $4 \nmid n$ or $t\not\equiv 1[4]$. If $p \nmid  2n$, then $\lambda_p(P)=0$.
\end{lem }
\begin{proof}
Let $p \mid D$. We assume that $p\nmid  n$ and $p\neq 2$.
First we have $4A^2=B^2(9\alpha+3t)+4\delta(\alpha^2-n^2 \alpha+n^4)$. Since $\delta$ is squarefree, we obtain that $p \mid (\alpha^2-n^2 \alpha+n^4)$. We also have
$$B^2=4(\alpha^2-n^2 \alpha+n^4)(\alpha+t+n^2)-4n^2(3\alpha+t)$$
so $p \mid (3\alpha+t)$. Furthermore $p$ divides the resultant of $A$ and $B^2$, viewed as polynomials in $\Z[\alpha]$, which is equal to $\Delta=16n^4 \delta^2$. Therefore $p \mid \delta$. 
Moreover,
$$27B^2=4(3\alpha+t)^3-4\delta(9\alpha+t-3n^2),$$
which implies that $p \mid (9\alpha+t-3n^2)$. Since $(9\alpha+t-3n^2)=3(3\alpha+t)-(2t+3n^2)$, we obtain that $p \mid (2t+3n^2)$. But $4\delta=(2t+3n^2)^2+27n^4$. We get $p=3$ hence $\delta$ is not squarefree, which is a contradiction. 
\end{proof}
\begin{prop}\label{loctr}
Let $P=(\alpha,\beta)$ be an integral point on $E$ such that $\ell P=(0,n^3)$ for some integer $\ell \geq 1$.  Then the following inequality holds
$$\Sum_{p <\infty} \lambda_p(P) \geq -\dfrac{1}{2}\log(n)-\dfrac{1}{3}\log(2).$$
\end{prop}
\begin{proof}
By Lemma \ref{canh1}, it suffices to compute the local contributions at primes $p \mid 2n$.
For $p \mid n$, we first assume $p \neq 2$.
Let $m=v_p(n)$. Since  $\delta $ is squarefree, we have $p \nmid c_4$ and $v_p(\Delta)=4m$. 
Put $m_p=\min(2m,v_p(B))$. By Definition/Proposition~\ref{locnarch}, the local contribution of $P$ is then  given by
$$
\lambda_p(P)=-\dfrac{m_p(4m-m_p)}{8m}\log(p).
$$ 
We easily conclude that
$$
\begin{array}{lll}
\lambda_p(P) &\geq -\dfrac{m}{2}\log(p) 
&\geq -\dfrac{1}{2}\log(p^m).
\end{array}
$$

Now we want to compute the $2$-local contribution $\lambda_2(P)$. Since $2\mid c_4$, it is given by
$$
\lambda_2(P)=
\left\{
\begin{array}{ll}
-\dfrac{v_2(B)}{3}\log(2) & $if $v_2(C) \geq 3v_2(B), \vspace{0.15cm}\\
-\dfrac{v_2(C)}{8}\log(2) & $otherwise$.
\end{array}
\right.
$$
If $2 \nmid n$ then $v_2(B)=1$, so $\lambda_2(P) \geq -\dfrac{1}{3} \log(2)$. \\
Suppose now that $2 \mid n$ and  $m=v_2(n)$. \\
If $m \in \{1,2\}$ it is easy to check that $\lambda_2(P) \geq -\dfrac{1}{2}\log(2^m)-\dfrac{1}{3}\log(2)$. So we may assume that $m \geq 3$.\\
If $v_2(\alpha) < m$ then $v_2(B)=1+v_2(\alpha) $; in that case we have $\lambda_2(P) \geq \dfrac{-m}{3}\log(2)$ if $v_2(C)  \geq 3v_2(B)$, else $\lambda_2(P) \geq -\dfrac{m}{2}\log(2)$. \\
If $m \leq v_2(\alpha) < \frac{4m-2}{3}$ then $v_2(B)= 1+v_2(\alpha)$ and $v_2(C)= 3 v_2(\alpha)+2$;
in that case we have $\lambda_2(P)=-\dfrac{3 v_2(\alpha)+2}{8} \log(2) \geq -\dfrac{1}{2}\log(2^m)$. \\
If $v_2(\alpha)=\dfrac{4m-2}{3}$ which is possible only if $m \equiv 2[3]$ then $v_2(C)>4m$ and $v_2(B)=\dfrac{4m+1}{3}$. In that case we have  $\lambda_2(P)=-\dfrac{4m+1}{9}\log(2) \geq -\dfrac{m}{2}\log(2)$. \\
If $v_2(\alpha) >\dfrac{4m-2}{3}$, then $v_2(C)=4m$ and $3v_2(B)> 4m$; in that case we have $\lambda_2(P) \geq -\dfrac{1}{2}\log(2^m)$.

Finally we conclude by summing all local non-archimedean contributions.
\end{proof}
Now we study the local archimedean contribution. 
\begin{defp}\label{defp-htlocalefinie}(\citep[7.5.7]{cohen2013course})
Let $P=(\alpha,\beta) \in E(\Q)$. Let $z$ be the elliptic logarithm of $P$. Let $\mu =\dfrac{2\pi}{\omega_1}$, $s=\mu \mathrm{Re}(z)$, $q= \exp\left(\dfrac{2i\pi \omega_2}{\omega_1}\right)$, and 
$$\theta=\Sum_{k=0}^{\infty}\sin((2k+1)s)(-1)^kq^{\frac{k(k+1)}{2}}.$$
Then the local archimedean contribution is given by
\begin{equation*}
\lambda_{\infty}(P)=\dfrac{1}{32}\log\left|\dfrac{\Delta}{q}\right|-\dfrac{1}{4}\log|\theta|+\dfrac{1}{8}\log\left|\dfrac{\alpha^3+\frac{b_2}{4}\alpha^2+\frac{b_4}{2}\alpha+\frac{b_6}{4}}{\mu}\right|
\end{equation*}
where $b_2=a_1^2+4a_2$, $b_4=a_1a_3+2a_4$, $b_6=a_3^2+4a_6$ and $a_1,a_2,a_3,a_4$ and $a_6$ are defined as in Lemma \ref{entier}.
\end{defp}
\begin{prop}\label{inqeqlocarch}
Assume that $4 \nmid n$ or $t \not\equiv 1[4]$ and that $\ell P=(0,n^3)$, for some integer $\ell$. If $t \geq \mathrm{max}(100n^2,n^4)$, we have 
$$
\lambda_\infty(P) \geq \dfrac{13}{80}\log(t)+\dfrac{3}{8}\log(n)+0.30.
$$
If $t \leq \mathrm{min}(-100n^2,-2n^4)$, we have
$$
\lambda_\infty(P) \geq \dfrac{3}{16} \log(|t|)+\dfrac{3}{8} \log (n)+0.27.
$$
\end{prop}
\begin{proof}
We first note that 
\begin{equation}\label{eq:majotheta}
|\theta|\leq \Sum_{k=0}^{\infty}q^{\frac{k(k+1)}{2}} \leq \dfrac{1}{1-q}.
\end{equation}
To find a lower bound for $\lambda_{\infty}(P)$, we need to find an upper bound for $q$. If $t\geq 100n^2$, by Lemmas \ref{l1} and \ref{l2}, we have 
$$ \dfrac{2i\pi \omega_2}{\omega_1}\leq \dfrac{-3.11 \times 2\pi}{5.35+1.23\log\left( \frac{t}{n^2}\right) }\leq \dfrac{-15.88}{\log(\frac{t}{n^2})}
$$
thus
\begin{equation}\label{qpos}
q \leq \exp \left(\dfrac{-15.88}{\log(\frac{t}{n^2})}\right) \leq 1-\dfrac{4.3}{\log(\frac{t}{n^2})}.
\end{equation}
On the one hand, by definition of $\Delta$ we have
\begin{align*}
\dfrac{1}{32}\log(\Delta)=\dfrac{1}{32} \log \big(16n^4(t^2+3n^2t+9n^4)^2 \big) \geq \dfrac{1}{8}\log(t)+\dfrac{1}{8}\log(n)+\dfrac{1}{8}\log(2). \end{align*}
On the other hand, with \eqref{qpos} and \eqref{eq:majotheta} we get
\begin{eqnarray*}
\dfrac{1}{32}\log \left|\dfrac{1}{q}\right| 
\geq  \dfrac{1}{32}\log \left(\exp\left( \dfrac{15.88}{\log(\frac{t}{n^2})}\right)\right) \geq  \dfrac{15.88}{32 \log(\frac{t}{n^2})}, 
\end{eqnarray*}
\begin{eqnarray*}
-\dfrac{1}{4}\log |\theta | \geq -\frac{1}{4}\log \left|\dfrac{1}{1-q} \right| &\geq & \dfrac{1}{4} \log(4.3)-\dfrac{1}{4}\log\log \left(\frac{t}{n^2}\right),
\end{eqnarray*}
and by Lemma \ref{l2},
$$
\begin{array}{ll}
-\dfrac{1}{8}\log (\mu) =\dfrac{1}{8}\log \left(\dfrac{\omega_1}{2 \pi}\right) &\geq -\dfrac{1}{8}\log(2\pi)+\dfrac{1}{8}\log \left(\dfrac{1.88+0.99\log \big(\frac{t}{n^2}\big)}{\sqrt{t}} \right), \vspace{0,25cm} \\
 &\geq -\dfrac{1}{16} \log (t)-\dfrac{1}{8} \log (2 \pi)+\dfrac{1}{8}\log \log \left(\dfrac{t}{n^2} \right)+\dfrac{1}{8}\log(0.99).
\end{array}
$$
Moreover, we have
$$
\dfrac{1}{8} \log(2)+\dfrac{1}{4}\log(4.3)-\dfrac{1}{8} \log (2 \pi)+\dfrac{1}{8} \log(0.99) \geq 0.22.
$$
Finally using Definition/Proposition~\ref{defp-htlocalefinie} and the fact that $\alpha^3+\frac{b_2}{4}\alpha^2+\frac{b_4}{2}\alpha+\frac{b_6}{4}=\beta^2$, we obtain the lower bound for $\lambda_{\infty}(P)$:
\begin{eqnarray*}
\lambda_{\infty}(P) &\geq & \dfrac{1}{16} \log(t)-\dfrac{1}{8} \log\log \left(\dfrac{t}{n^2} \right)+\dfrac{15.88}{32 \log(\frac{t}{n^2})} \\
& &+\dfrac{1}{4} \log |\beta|+\dfrac{1}{8} \log (n)+0.22.
\end{eqnarray*}
Moreover, by Lemma \ref{b+}, for $t \geq n^4$ we have $|\beta|\geq n\sqrt{2t}$. If $t \geq \max(n^4,100n^2)$, we obtain
\begin{equation*}
\lambda_{\infty}(P) \geq \dfrac{13}{80}\log(t)+\dfrac{3}{8}\log(n)+0.30.
\end{equation*}
When $t$ is negative, we proceed similarly to find a lower bound to $\lambda_{\infty}(P)$. If $t\leq -100n^2$, we have by Lemmas~\ref{l3} and \ref{l4} 
$$ \dfrac{2i\pi \omega_2}{\omega_1}\leq \dfrac{-2\pi \left(0.39 +\log\Big(\dfrac{|t|}{n^2}\Big)\right)}{3.15 }\leq -\dfrac{2\pi}{3.15}\log\left(\frac{|t|}{n^2}\right),
$$
and
\begin{equation}\label{qneg}
q \leq \exp \left( -\dfrac{2 \pi}{3.15}\log \Big(\frac{|t|}{n^2}\Big)\right) \leq 0.0002.
\end{equation}
On the one hand, by definition of $\Delta$ we have
\begin{eqnarray*}
\dfrac{1}{32}\log(\Delta)
&\geq & \dfrac{1}{8}\log(|t|)+\dfrac{1}{8}\log(n)+\dfrac{1}{8}\log(2)+\frac{1}{32}\log \left(1- \frac{3}{100}\right)
\end{eqnarray*}
On the other hand, with \eqref{qneg} and \eqref{eq:majotheta} we get 
\begin{eqnarray*}
\dfrac{1}{32}\log \left|\dfrac{1}{q}\right| &\geq &  \dfrac{1}{32}\log \exp\left( \frac{2\pi}{3.15}\log \left(\frac{|t|}{n^2}\right)\right) \geq 0.28,
\\
-\dfrac{1}{4}\log |\theta | &\geq &- \dfrac{1}{4}\log \left|\dfrac{1}{1-q} \right| \geq -\dfrac{1}{4} \log \left(\dfrac{1}{1-0.0002}\right),
\end{eqnarray*}
and, by Lemma~\ref{l3},
$$
\begin{array}{ll}
-\dfrac{1}{8}\log (\mu) &\geq -\dfrac{1}{8}\log(2\pi)+\dfrac{1}{8}\log \left(\dfrac{3.14}{\sqrt{|t|}} \right) \vspace{0,25cm} \\ 
&\geq -\dfrac{1}{16} \log \left(|t| \right)-\dfrac{1}{8} \log (2 \pi)+\dfrac{1}{8}\log (3.14).
\end{array}$$ 
Moreover we have
$$
\dfrac{1}{8} \log(2)+\dfrac{1}{8}\log(3.14)-\dfrac{1}{8} \log (2 \pi)+\dfrac{1}{4} \log \left(1-0.0002\right)+0.28\geq 0.27.
$$
Then we obtain
$$
\lambda_{\infty}(P) \geq \dfrac{1}{16} \log(|t|)+\dfrac{1}{4} \log |\beta|+\dfrac{1}{8} \log (n)+0.27.
$$
Moreover by Lemma \ref{b-}, for $t \leq -2n^4$, we have $|\beta| \geq n\sqrt{|t|}$. If $t \leq \min (-2n^4,-100n^2)$, we obtain
\begin{equation*}
 \lambda_{\infty}(P) \geq \dfrac{3}{16} \log(|t|)+\dfrac{3}{8} \log (n)+0.27.
 \end{equation*}
\end{proof}

 \begin{theo}\label{theo-minohauteurP}
Assume that $4 \nmid n$ or $ \not\equiv 1[4]$.
Let $P$ be an integral point on $E$ such that $\ell P=(0,n^3)$ for some integer $\ell \geq 2$. If $t \geq \max(100n^2,n^4)$, we have
$$
 \widehat{h}(P)\geq \dfrac{13}{80}\log(t)-\dfrac{1}{8}\log(n)+0.09.
$$
If $t \leq \min (-100n^2,-n^4)$, we have 
$$
\widehat{h}(P) \geq \dfrac{3}{16} \log(|t|)-\dfrac{1}{8} \log (n).
$$

\end{theo}
\begin{proof}
It suffices to sum the inequalities obtained in  Proposition \ref{loctr} and Proposition \ref{inqeqlocarch}. 
\end{proof}
By similar arguments, we obtain the following statements for the case $t=4k+1$ and $n=4m$.
\begin{prop}\label{mlocp}
If $P=(\alpha,\beta)$ is an integral point of $E'$ such that $\ell P=(0,8m^3)$ for some integer $\ell \geq 1$, then the following inequality holds 
$$
\Sum_{p <\infty}\lambda_p(P) \geq -\frac{1}{2}\log(m)\geq -\frac{1}{2}\log(n)+\log(2).
$$
\end{prop}
\begin{prop}\label{mlocinf}
 Let $P \in E'(\Q)$ be an integral point such that $\ell P=(0,8m^3)$ for some integer $\ell \geq 2$. If $t \geq \max(n^4,100n^2)$, we have 
$$
\lambda_{\infty}(P) \geq \dfrac{13}{80}\log(t)+\dfrac{3}{8}\log(n)-\dfrac{5}{8}\log(2)+0.04.
$$
If  $t \leq \min(-100n^2,-2n^4)$, we have 
$$
\lambda_{\infty}(P) \geq \dfrac{3}{16}\log(|t|)+\dfrac{3}{8}\log(n)-\dfrac{3}{4}\log(2)+0.10.
$$
\end{prop}
Again after summing the inequalities obtained in Propositions \ref{mlocp} and \ref{mlocinf}, we obtain the following statement.
\begin{theo}\label{minhaunm}
 Let $P$ be an integral point on $E'$ such that $\ell P=(0,8m^3)$, for some integer $\ell \geq 2$. If $t \geq \max(100n^2,n^4)$, we have 
$$
 \widehat{h}(P)\geq \dfrac{13}{80}\log(t)+\dfrac{3}{8}\log(n)+0.04.
$$
If $t \leq \min(-100n^2,-2n^4)$, we have 
$$
\widehat{h}(P) \geq \dfrac{3}{16} \log(|t|)+\dfrac{1}{4} \log (n)+0.10.
$$
\end{theo}

\subsection{Upper bound for the height}

\begin{prop}\label{prop-majohauteur}
If $|t| \geq 100n^2$, we have the following inequality
$$\widehat{h}((0,n^3)) \leq  \log(|t|)+1.57  .$$
\end{prop}
\begin{proof}
We use the result of Silverman \cite[Theorem~1.1]{silverman1990difference}. Let $P \in E(\Q)$. We have 
$$ \widehat{h}(P)-\frac{1}{2}h(P)\leq \dfrac{1}{12}h(\Delta)+\dfrac{1}{12}h_{\infty}(j)+ \dfrac{1}{2}h_{\infty}\left(\dfrac{b_2}{12} \right)+\dfrac{1}{2}\log(2)+1.07$$
where $h(P)=h(x(P))$ is the logarithmic height over $\Q$, $h_{\infty}(x)=\max(\log|x|,0)$ and $b_2$ is defined as in Definition/Proposition \ref{defp-htlocalefinie}.
Since $h((0,n^3))=0$, in our case we have 
\begin{eqnarray*}
\widehat{h}((0,n^3)) &\leq & \dfrac{1}{12}\log(16n^4(t^2+3n^2t+9n^4)^2)+\dfrac{1}{12}\log \left(\frac{256}{n^4}(t^2+3n^2t+9n^4)\right) \\
& & \ +\frac{1}{2}\log \left(\frac{|t|}{3}\right)+\frac{1}{2}\log(2)+1.07\\
&\leq & \dfrac{1}{4}\log(t^2+3n^2t+9n^4)+\dfrac{1}{2}\log(|t|)+ 1.561 \\
& \leq & \log(|t|)+\dfrac{1}{4}\log \left(1+\frac{3n^2}{|t|}+\frac{9n^4}{t^2}\right)+1.561.
\end{eqnarray*}
 Finally, if $|t| \geq 100n^2$, we obtain
\begin{equation*}
\widehat{h}((0,n^3))\leq \log(|t|)+1.57.
\end{equation*}
\end{proof}
\begin{rem}
We note that the upper bound of Proposition \ref{prop-majohauteur} is independent of $n$. In fact, this bound is not optimal.
A numerical analysis suggests that $\widehat{h}((0,n^3))$ should be equivalent to $\frac{1}{2} \log \left(\frac{|t|}{n^2}\right)$ as $|t| \rightarrow +\infty$.
\end{rem}
By an argument similar to Proposition \ref{prop-majohauteur}, we obtain the following statement.
\begin{prop}\label{upphnm}
For the case $t=4k+1$, $n=4m$. If $|t| \geq 100n^2$, we have 
$$
\widehat{h}((0,8m^3))\leq \log(|t|)+0.19.
$$
\end{prop}

\section{Main results}\label{Pthr}
\begin{prop}\label{th1}
Suppose that $t \geq \max(100n^2,n^4)$ or $t \leq \min(-100n^2,-2n^4)$. Suppose also that the equation $y^2=f(x)$ is a minimal Weierstrass model for $E$ when $\delta$ is squarefree (which occurs if $t \not\equiv 1\, [4]$ when $4\mid n$). Then the point $(0,n^3)$ is not divisible. 
\end{prop}
\begin{proof}
Assume that $t$ is positive. Let $P$ be a point in $E(\Q)-E_0(\Q)$ and $\ell \geq 2$ such that $\ell P=(0,n^3)$. 
By Lemma~\ref{entier}, $P$ must be an integral point. By Theorem~\ref{theo-minohauteurP} we know that $\widehat{h}(P) \geq  \dfrac{13}{80}\log(t)-\dfrac{1}{8}\log(n)+0.06$ 
and by Proposition~\ref{prop-majohauteur}, we have $\widehat{h}((0,n^3))\leq \log(t)+1.57$. We note that $P$ has infinite order and then $\widehat{h}(P) \neq 0$.
Since  $\widehat{h}$ is quadratic, we have
$$ 
\begin{array}{ll}
\ell^2=\dfrac{\widehat{h}((0,n^3))}{\widehat{h}(P)} &\leq \dfrac{\log(t)+1.57}{\frac{13}{80}\log(t)-\frac{1}{8}\log(n)+0.06} \vspace{0.2cm}\\
& \leq \dfrac{80}{13} +\dfrac{\frac{2}{3}\log(n)+1.21}{\frac{13}{80}\log(t)-\frac{1}{8}\log(n)+0.06} \vspace{0,2cm} \\ &\leq 8.6
\end{array}
$$
thus $\ell \leq 2$. However we have seen at the beginning of Section~\ref{subs:lowerbound} that $\ell$  is odd, hence the point $(0,n^3)$ is not divisible.

Assume now that $t$ is negative. We have by a similar argument (see Theorem~\ref{theo-minohauteurP} and  Proposition~\ref{prop-majohauteur}),
\begin{align*}
\ell^2 & \leq  \dfrac{\log(|t|)+1.57}{\frac{3}{16}\log(|t|)-\frac{1}{8}\log(n)} \vspace{0.9cm} \\
& \leq 5.34+ \dfrac{\frac{2}{3}\log(n)+1.57}{\frac{3}{16}\log(|t|)-\frac{1}{8}\log(n)} \\
& \leq  8.85.
\end{align*}
Again we get $\ell \leq 2$ which is impossible. Thus the point $(0,n^3)$ is not divisible.
\end{proof}
Using Theorem \ref{minhaunm} and Proposition \ref{upphnm}, we obtain the following statement.
\begin{prop}\label{enotmin}
We assume $n=4m$ and $t=4k+1$, for some $m \in \Z_{>0}$ and $k \in \Z$. If $\delta$ is squarefree, and $t \geq \max(100n^2,n^4)$ or $t \leq \min(-100n^2,-2n^4)$, then the point $(0,n^3)$ is not divisible on $E$.
\end{prop}
By combining Propositions \ref{th1} and \ref{enotmin} , we obtain the  main result.
\begin{theo}\label{mainth}
Suppose that $t \geq \max(100n^2,n^4)$ or $t \leq \min(-100n^2,-2n^4)$, and $\delta$ is squarefree. Then the point $(0,n^3)$ is not divisible on $E$.
\end{theo}
Now, to extend Theorem~\ref{mainth} to the case $|t|< \max(100n^2,n^4)$, we can use the lower bound for the height in \citep[Proposition 2.1]{MR1864622}: for all $P \in E(\Q)$ of infinite order, if $P$ is non-singular modulo $p$ for all prime $p$, we have
\begin{equation*}
\widehat{h}(P) > \dfrac{1}{12N^2}\log|\Delta_{\min}|,
\end{equation*}
where $N$ is the number defined in \citep[Theorem 1]{MR1864622}, and $\Delta_ {\min}$ is the minimal discriminant of $E$. Let $C_E$ be the lowest common multiple of Tamagawa numbers of $E$, and $P \in E(\Q)$ a point of infinite order, the point $C_E P$ is non-singular modulo $p$ for all prime $p$ so we have
$$
\widehat{h}(P) > \dfrac{1}{12N^2C_E^2}\log|\Delta_{\min}|
$$
 Since $\Delta >0$, we have $N=6$ or $N=8$. Hence we have for all $P \in E(\Q)$ of infinite order
$$ \widehat{h}(P) > \dfrac{1}{768 C_E^2}\log|\Delta_{\min}|.$$
 Let us denote by $B_E$ this lower bound.
So for a given $t$, it suffices to check that $(0,n^3)$ is not equal to $\ell P$ for all 
$P \in E(\Q)$ and for all primes $\ell \leq \sqrt{\dfrac{\widehat{h}(0,n^3)}{B_E}}$.

We use the following method: if there exist a prime $\ell$ and a point $P \in E(\Q)$ such that $\ell P=(0,n^3)$ then for all primes $p \nmid \Delta$, the reduction of $E$ modulo $p$ is an elliptic curve over $\mathbb{F}_p$ and $\ell \overline{P}=\overline{(0,n^3)}$, where $\overline{P}$ denote the reduction of the point $P$. Since $p \nmid \Delta$ and  the point $P$ is an integral point, the reduction modulo $p$ of $P$ is well-defined and is a point of $E(\F)$. If  $p$ is a prime number such that $\ell$ divides the exponent of group $E(\mathbb{F}_p)$ denoted by $r_p$, we shall have $\dfrac{r_p}{\ell}\overline{(0,n^3)}=O$. So, if we find a prime $p$ such that $\ell \mid r_p$ and $\frac{r_p}{\ell}\overline{(0,n^3)}\neq O$, there does not exist $P \in E(\Q)$ such that $\ell P=(0,n^3)$.

We obtain the following statement.
\begin{theo}
If $n \leq 10$ and $\delta$ is squarefree, then the point $(0,n^3)$ is not divisible.
\end{theo}
\begin{proof}
Assume $n=1$. Theorem~\ref{mainth} implies that if $|t| \geq 100$, the point is not divisible. We also note that the result is true for negative $t$ because $(0,1)$ is the only integer point in $E(\Q)-E_0(\Q)$. For $1\leq t <100$, it suffices to use the method described above with the help of PARI/GP \cite{PARI2}.
For $2\leq n \leq 10$, we use a similar method.
\end{proof}

\newpage
\bibliographystyle{apalike}
\bibliography{bibliographie}
\addcontentsline{toc}{section}{References}
\vspace{1cm}
Valentin Petit \\
valentin.petit@univ-fcomte.fr \\
Laboratoire de mathématiques de Besançon \\
     Université Bourgogne Franche-Comté \\
     CNRS UMR 6623 \\
     16, route de Gray \\
     25030 Besançon Cedex\\
     France
\end{document}